\documentclass{proc-l}

\usepackage{amsmath, amssymb, latexsym, amsthm,bm}
\usepackage{enumitem}
\usepackage{mathrsfs}
\usepackage{color}
\usepackage[usenames,dvipsnames]{xcolor}
\usepackage{graphicx}
\usepackage{stmaryrd,amsbsy,enumerate}

\usepackage{float}
\usepackage{caption}

\usepackage{amsrefs}
\usepackage{array}

\usepackage{stmaryrd}

\def\Tinv{\reflectbox{\rotatebox[origin=c]{180}{$T$}}}

\def\Ds{{\,\bigl|\,}}

\def\lk{{\text{\rm lk}}}
\def\Lk{{\text{\rm Lk}}}

\numberwithin{equation}{section}

\DeclareMathAlphabet{\mato}{U}{bbm}{m}{sl}

 \DeclareMathAlphabet{\foo}{OT1}{cmtt}{m}{n}

\definecolor{darkgreen}{rgb}{0.0, 0.2, 0.13}

\definecolor{darkgreen}{rgb}{0.0, 0.5, 0.0}

\def\mea#1{{\llbracket{#1}\rrbracket}}

\def\somea{{\Scale[1.5]{\text{\rm (a)}}}}
\def\someb{{\Scale[1.3]{\text{\rm (b)}}}}

\def\dointa{{\Scale[2.1]{(-1,3)}}}
\def\dointb{{\Scale[2.1]{(1,3)}}}
\def\dointc{{\Scale[2.1]{(-1,0)}}}
\def\dointd{{\Scale[2.1]{(1,0)}}}

\def\gointa{{\Scale[1]{(-1,3)}}}
\def\gointb{{\Scale[1]{(1,3)}}}
\def\gointc{{\Scale[1]{(-1,0)}}}
\def\gointd{{\Scale[1]{(1,0)}}}

\def\hointa{{\Scale[1.2]{(-1,3)}}}
\def\hointb{{\Scale[1.2]{(1,3)}}}
\def\hointc{{\Scale[1.2]{(-1,0)}}}
\def\hointd{{\Scale[1.2]{(1,0)}}}

\def\sointA{{\Scale[1.6]{a}}}
\def\sointB{{\Scale[1.6]{b}}}

\def\sointt{{\Scale[1.5]{T_0}}}

\def\ignore#1{{ }}

\definecolor{mygray}{gray}{0.8}

\def\mm{{\mathscr M}}

\def\cc{{\mathscr C}}

\def\myfrac#1#2{{\genfrac{}{}{0pt}{}{#1}{#2}}}

\def\cir#1#2#3{{\cc{\hbox{\hglue #1 cm}\myfrac{#2}{#3}}}}

\def\centerarc[#1](#2)(#3:#4:#5){ \draw[#1] ($(#2)+({#5*cos(#3)},{#5*sin(#3)})$) arc (#3:#4:#5); } 

\def\conethree{\cir{-0.04}{1}{3}}

\def\cthreefive{\cir{-0.02}{3}{5}}

\def\c13{{\conethree}}
\def\c35{{\cthreefive}}

\newtheorem{theorem}{Theorem}[section] 
\newtheorem{theorem*}{Theorem}[section] 
\newtheorem{proposition}[theorem]{Proposition} 

\newtheorem{lemma}[theorem]{Lemma}
\newtheorem{fact}[theorem]{Fact}
\newtheorem{conjecture}[theorem]{Conjecture}
\newtheorem{observation}[theorem]{Observation}

\theoremstyle{definition}
\newtheorem{definition}{Definition}[section]
\newtheorem{question}[definition]{Question}

\newcommand*{\Scale}[2][4]{\scalebox{#1}{\ensuremath{#2}}}%

\def\dnw{{\Scale[0.9]{(-1,1,0)}}}
\def\dne{{\Scale[0.9]{(1,1,0)}}}
\def\dsw{{\Scale[0.9]{(-1,-1,0)}}}
\def\dse{{\Scale[0.9]{(1,-1,0)}}}

\def\dnw{{\Scale[0.9]{(-1,1,0)}}}
\def\dne{{\Scale[0.9]{(1,1,0)}}}
\def\dsw{{\Scale[0.9]{(-1,-1,0)}}}
\def\dse{{\Scale[0.9]{(1,-1,0)}}}

\def\dnw{{\Scale[1.1]{(-1,1)}}}
\def\dne{{\Scale[1.1]{(1,1)}}}
\def\dsw{{\Scale[1.1]{(-1,-1)}}}
\def\dse{{\Scale[1.1]{(1,-1)}}}

\def\sointt{{\Scale[1.4]{T}_{\Scale[1.1]{0}}}}

\def\Sonetu{{\Scale[1.4]{S}_{\Scale[0.9]{1}}}}
\def\Stwotu{{\Scale[1.4]{S}_{\Scale[0.9]{2}}}}

\def\Lone{{\Scale[1.4]{L}}}

\def\Rone{{\Scale[1.4]{R}}}
\def\aDelta{{\Scale[1.2]{\Delta}}}

\def\ALone{{\Scale[1.4]{A}}}
\def\BLone{{\Scale[1.4]{B}}}

\def\tbigaone{{\Scale[1.1]{1}}}
\def\tbigam{{\Scale[1.1]{m}}}

\def\bigazero{{\Scale[1.4]{0}}}
\def\bigam{{\Scale[1.2]{m}}}

\def\Cm1{{\Scale[1.3]{\text{$m$ crossings, $m$ even}}}}
\def\Cm2{{\Scale[1.2]{\text{$m$ crossings, $m$ odd}}}}

\def\darena{{\Scale[2.8]{a}}}
\def\darenB{{\Scale[2.8]{b}}}
\def\darenb{{\Scale[2.8]{\text{(b)}}}}
\def\earenb{{\Scale[1.6]{\text{(b)}}}}
\def\tarena{{\Scale[1.5]{\text{(a)}}}}

\subjclass[2010]{Primary 57M25}

\begin{document}



\title{\bf The knots that lie above all shadows} 

\author{Carolina Medina}
\address{Department of Mathematics, University of California, Davis, CA 95616, USA}
\thanks{The first author was supported by Fordecyt grant 265667, and is currently supported by a Fulbright Visiting Scholar Grant at UC Davis. The second author was supported by Conacyt grant 222667 and by FRC-UASLP}
\email{cmedina@math.ucdavis.edu}

\author{Gelasio Salazar}
\address{Instituto de F\'\i sica, Universidad Aut\'onoma de San Luis Potos\'{\i}, SLP 78000, Mexico}
\email{gsalazar@ifisica.uaslp.mx}





\maketitle

\def\K{{\#}}

\begin{abstract}
We show that for each even integer $m\ge 2$, every reduced shadow with sufficiently many crossings is a shadow of a torus knot $T_{2,m+1}$, or of a twist knot $T_m$, or of a connected sum of $m$ trefoil knots. 
\end{abstract}

\section{Introduction}\label{sec:intro}

A {\em shadow} is an orthogonal projection of a knot onto a plane. The {\em size} of a shadow is its number of crossings. As usual, all shadows under consideration are {\em regular}, that is, they have no triple points and no points of self-tangency. A shadow $S$ {\em resolves into} a knot $K$ if there is an over/under assignment at the crossings of $S$ that gives a diagram of $K$. 

This work revolves around the following fundamental problem. 

\begin{question}\label{que:que1}
Given a shadow  $S$, which knots $K$ satisfy that $S$ resolves into $K$? 
\end{question}

To investigate this question we must restrict our attention to {\em reduced} shadows, that is, shadows with no nugatory crossings. As in~\cite{taniyama}, we say that a crossing $x$ in a shadow $S$ is {\em nugatory} if $S{\setminus}\{x\}$ is disconnected.  This restriction is crucial: it is easy to exhibit arbitrarily large non-reduced shadows that only resolve into the unknot.

In Figure~\ref{fig:one}(a) we illustrate the shadow of a minimal crossing diagram of a torus knot $T_{2,m+1}$, {and in (b) we show the shadow of a minimal crossing diagram of a twist knot $T_m$.}  As proved in~\cite{fertility}, these shadows only resolve into torus knots $T_{2,n}$ and into twist knots, respectively. 

\begin{figure}[ht!]
\centering
\vglue 0.3 cm
\scalebox{0.5}{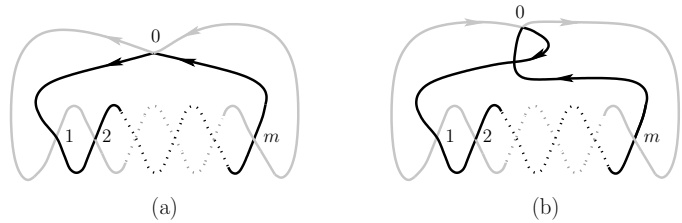}
\caption{The shadow of a minimal crossing diagram of a torus knot $T_{2,m+1}$ (left) and of a twist knot $T_m$ (right).}
\label{fig:one}
\end{figure}

{Thus} there are arbitrarily large reduced shadows that only resolve into torus knots $T_{2,n}$ {(including the unknot $T_{2,1}$)}, and there are arbitrarily large reduced shadows that only resolve into twist knots {(including the unknot $T_0$).}

As we {illustrate} in Figure~\ref{fig:two}, {there are} arbitrarily large reduced shadows that only resolve into connected sums of (left-handed or right-handed) trefoil knots, and into the unknot. 

\begin{figure}[ht!]
\centering
\vglue 0.3 cm
\scalebox{0.5}{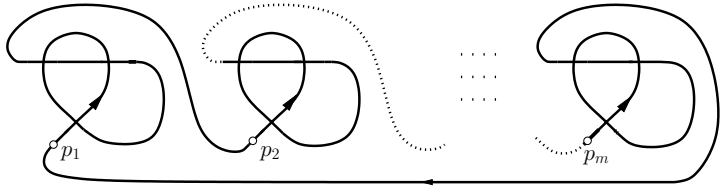}
\caption{This shadow only resolves into connected sums of trefoil knots. The labelled points are indicated for future reference.}
\label{fig:two}
\end{figure}

Torus knots $T_{2,n}$ and twist knots are the simplest prime knots and connected sums of trefoil knots are the simplest composite knots. Our main result is that these three knot types lie at the core of Question~\ref{que:que1}, in the sense that every reduced shadow resolves into a knot with ``large'' crossing number in one of these families.  

\begin{theorem}\label{thm:main}
For each even integer $m\ge 2$, there is an integer $n$ with the following property. Every reduced shadow with at least $n$ crossings  resolves either into a torus knot $T_{2,m+1}$, or into a twist knot $T_m$, or into a connected sum of $m$ trefoil knots. 
\end{theorem}

{We remark that throughout this work we do not distinguish between a knot and its mirror image. It is valid to take this license because clearly a shadow resolves into a knot if and only if it resolves into its mirror image.}


\ignore{
\begin{figure}[ht!]
\centering
\vglue 0.3 cm
\scalebox{0.5}{\input{nt12.pdf_t}}
\caption{The shadow of a minimal crossing diagram of a torus knot $T_{2,m+1}$ (left) and of a twist knot $T_m$ (right).}
\label{fig:one}
\end{figure}
}

\ignore{For each odd integer $n\ge 3$, $T_{2,n}$ is $n_1$ in Rolfsen knot table~\cite{atlas,rolfsen}. For each even integer $n\ge 4$, $T_{n-2}$ is the knot $n_1$, and for each odd $n\ge 5$, $T_{n-2}$ is $n_2$.}


\subsection{Related work}


Besides proving the result mentioned above on the shadows in Figure~\ref{fig:one}, Cantarella, Henrich, Magness, O'Keefe, Perez, Rawdon, and Zimmer investigate in~\cite{fertility} several problems related to Question~\ref{que:que1}, including an exhaustive analysis on shadows of minimal crossing diagrams of knots with crossing number at most $10$. 

In~\cite{hanaki1}, Hanaki investigates the following related question: given a shadow $S$, what can be said about invariants of a knot that projects to $S$? Hanaki's work illustrates very well the difficulty of Question~\ref{que:que1}, with a running example of a shadow with $9$ crossings for which it is not easy to determine whether or not it resolves into the torus knot $T_{2,7}$.



In a seminal paper, Taniyama~\cite{taniyama} proved that every nontrivial reduced shadow resolves into a trefoil knot, and characterized which shadows resolve into a figure-eight knot, or into a torus knot $T_{2,5}$, or into a twist knot $T_3$ ($5_2$ in Rolfsen's table). 

In~\cite{taniyama2}, Taniyama proved the following result closely related to Theorem~\ref{thm:main}. For each even integer $m\ge 2$, there is an integer $n$ such that the following holds. If $S$ is a $2$-component link shadow, in which the projections of the components cross each other at least $n$ times, then $S$ resolves into a torus link $T_{2,m}$. The techniques and ideas in Taniyama's proof are the workhorse of the proof of one of our key lemmas.


\section{Proof of Theorem~\ref{thm:main}}\label{sec:proofmain}

{We start with an informal account of the main ideas in the proof of Theorem~\ref{thm:main}. The proof has three main ingredients. Let $m\ge 2$ be an integer. First we identify four kinds of shadows. A shadow that shares certain features with the shadows in Figure~\ref{fig:one}(a) (respectively, (b)) will be called $m$-{\em consistent} (respectively, $m$-{\em reverse}). A shadow that shares certain features with the shadow in Figure~\ref{fig:two} will be called $m$-{\em nontrivial}.}

{
We identify a fourth kind of shadow, inspired by the type of shadow illustrated in Figure~\ref{fig:fir}. A shadow will be called $m$-{\em decomposable} if it can be ``decomposed'' into $m$ shadows, each of which resolves into a trefoil knot. 
}

\begin{figure}[ht!]
\centering
\scalebox{0.5}{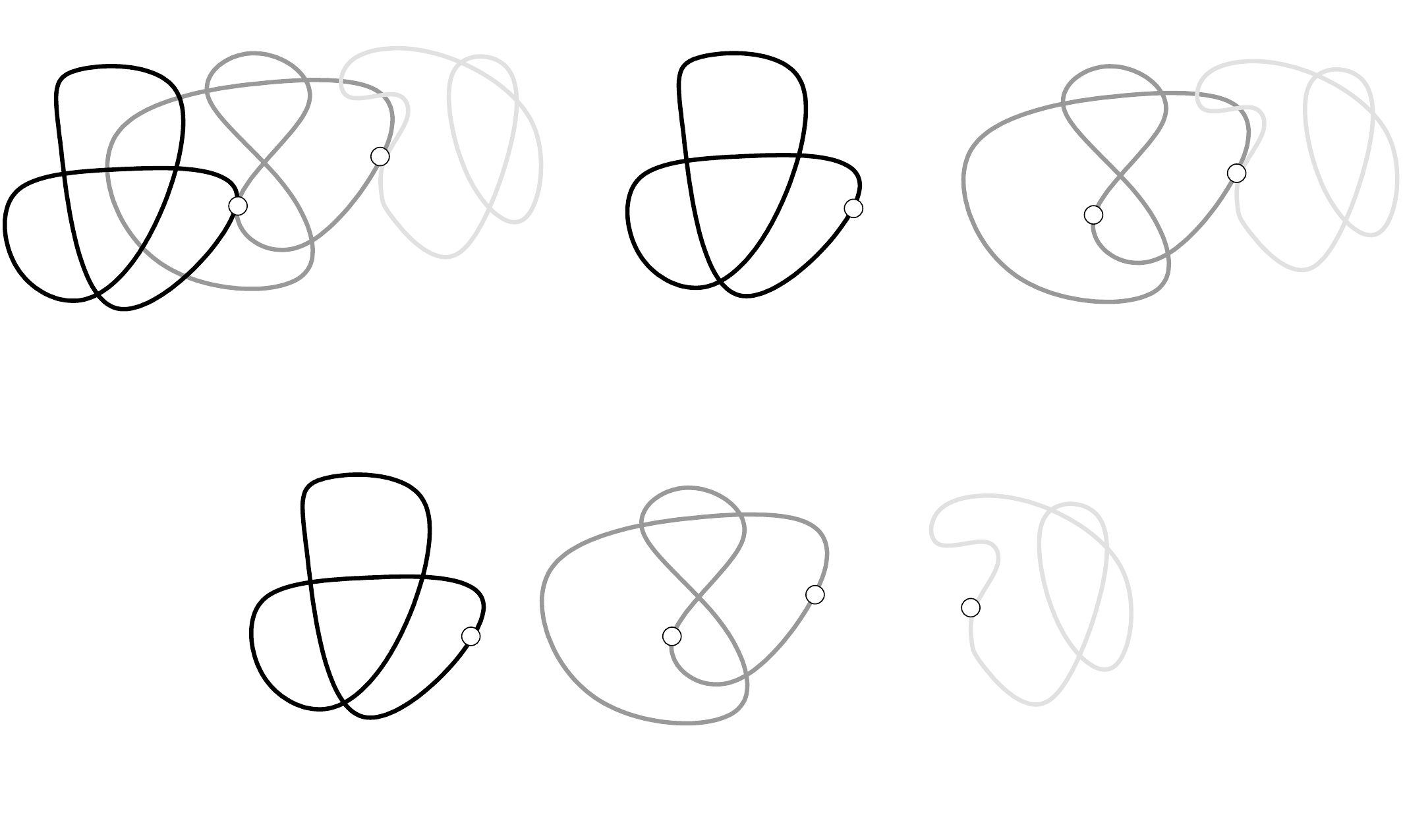}
\caption{A $3$-decomposable shadow.} 
\label{fig:fir}
\end{figure}

{The second main ingredient in the proof is that each of these four kinds of shadows resolves into a knot of one of the types in the statement of Theorem~\ref{thm:main}. More precisely, an $m$-consistent shadow resolves into a torus knot $T_{2,m+1}$, an $m$-reverse shadow resolves into a twist knot $T_m$, and an $m$-nontrivial or $m$-decomposable shadow resolves into a connected sum of $m$ trefoil knots.
}

{The third and final ingredient in the proof is the Structure Lemma: for each fixed even integer $m\ge 2$, every sufficiently large shadow is either $m$-consistent, or $m$-reverse, or $m$-nontrivial, or $m$-decomposable. In view of the previous paragraph, this completes the proof of Theorem~\ref{thm:main}.
}

\ignore{consists of showing that every sufficiently large reduced shadow $S$ either (i) ``resembles'' one of the shadows in Figures~\ref{fig:one} and~\ref{fig:two}, and so it resolves into a torus knot $T_{2,m+1}$, or into a twist knot $T_m$, or into a connected sum of trefoil knots; or (ii) can be ``decomposed'' into $m$ shadows, each of which resolves into a trefoil knot, and so $S$ resolves into a connected sum of trefoil knots.}

\subsection{{Subarcs and subshadows.}}
{Before we proceed to formally identify the four types of shadows mentioned in the previous subsection, we discuss the notions of a subarc and of a subshadow of a shadow.}

{We refer the reader to Figure~\ref{fig:Notions1}. Let $S$ be a shadow with a pre-assigned traversal direction. An {\em open subarc} of $S$ is obtained by taking two distinct points $x,y\in S$ that are not crossing points, and traversing $S$ from $x$ to $y$, following this direction.}

{Suppose now that $x$ is a crossing point. If we traverse $S$ starting at $x$ following its direction, until we come back to $x$, we obtain a {\em closed subarc} of $S$. If the closed subarc is a simple closed curve, then it is a {\em loop}, and $x$ is its {\em root}.}

{Note that for each crossing $x$ of $S$ there are two closed subarcs $S_1,S_2$ that start and end at $x$. Note also that $S_1$ and $S_2$ are shadows in their own right. We say that $S_1$ and $S_2$ are the {\em subshadows} of $S$ {\em based at $x$}, and write $S=S_1{\oplus_{x}} S_2$.}

\begin{figure}[ht!]
\centering
\scalebox{0.8}{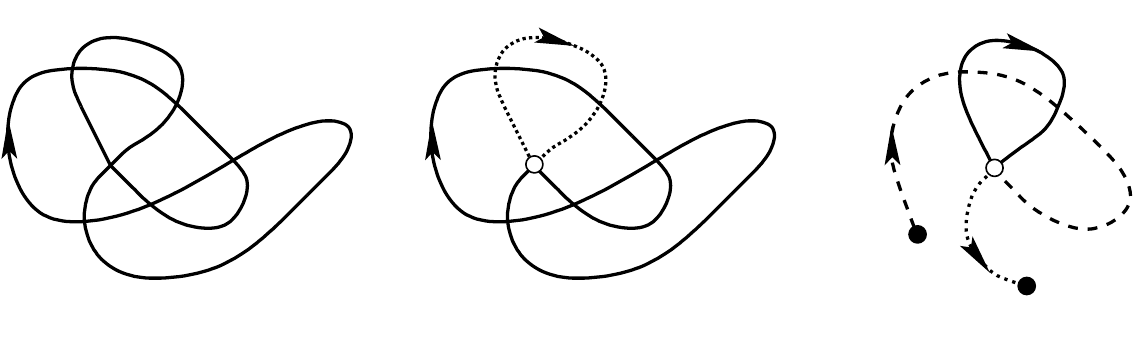}
\caption{In (a) we have a shadow $S$. In (b) we show two closed subarcs $L$ (dotted) and $M$ (solid) of $S$, where $L$ is a loop with root $x$. {Here $L$ and $M$ are the subshadows of $S$ based at $x$, and so $S=L{\oplus_{x}}M$.} In (c) we show an open nontrivial subarc of $S$.} 
\label{fig:Notions1}
\end{figure}

We remark that we assume that every shadow under consideration comes with a preassigned traversal direction. As illustrated in Figure~\ref{fig:Notions1}, a (open or closed) subarc of a shadow naturally inherits the traversal orientation of the shadow. 

\subsection{{Two kinds of shadows inspired by Figure~\ref{fig:one}}}
{We start by identifying the feature that we capture from the shadow in Figure~\ref{fig:one}(a).}

\begin{definition}
A shadow $S$ is $m$-{\em consistent} if it has a crossing $0$ such that the subshadows $S_1,S_2$ based at $0$ cross each other at points $1,2,\ldots,m$, and as we traverse each of $S_1$ and $S_2$ starting at $0$, we encounter these crossings precisely in this order.
\end{definition}

\begin{lemma}\label{lem:third}
For each even integer $m\ge 2$, every $m$-consistent shadow resolves into a torus knot $T_{2,m+1}$.
\end{lemma}

We defer the proof of this lemma to Section~\ref{sec:thirdfourth}. We remark that in this definition it is not required that $S_1$ and $S_2$ cross each other {\em only} at these $m$ points. They may cross each other arbitrarily many times, but as long as there exist $m$ crossing points with the required property, then $S$ is $m$-consistent. A similar remark applies to the following definition, motivated by the shadow of the twist knot in Figure~\ref{fig:one}(b). 
\begin{definition}
A shadow $S$ is $m$-{\em reverse} if it has a crossing $0$ such that the subshadows $S_1,S_2$ based at $0$ cross each other at points $1,2,\ldots,m$, and as we traverse $S_1$ starting at $0$ we encounter these crossings in this order, but as we traverse $S_2$ starting at $0$ we encounter them in the reverse order $m,\ldots,2,1$.
\end{definition}

\begin{lemma}\label{lem:fourth}
For each even integer $m\ge 2$, every $m$-reverse shadow resolves into a twist knot $T_m$.
\end{lemma}

We also defer the proof of this lemma to Section~\ref{sec:thirdfourth}. 

\subsection{{A kind of shadow inspired by Figure~\ref{fig:two}}}

{The shadow in Figure~\ref{fig:two} is the concatenation of $m$ open subarcs: the open subarcs that start at $p_i$ and end at $p_{i+1}$, for $i=1,\ldots,m-1$, and the one that starts at $p_m$ and ends at $p_1$.} 

{Each of these $m$ open subarcs has the following property, illustrated in Figure~\ref{fig:Notions1}(c): it can be written as a concatenation $\alpha L \beta$, where $L$ is a loop and $\alpha\cup\beta$ crosses $L$ at least twice. We say that an open arc with this property is {\em nontrivial}.}


\begin{definition}
A shadow is $m$-{\em nontrivial} if it is the concatenation of $m$ nontrivial open subarcs. 
\end{definition}

\begin{lemma}\label{lem:first}
For each integer $m\ge 1$, every $m$-nontrivial shadow resolves into a connected sum of $m$ trefoil knots.
\end{lemma}

\begin{proof}
In~\cite[Theorem 1]{taniyama} it is proved that every reduced shadow that is not a simple closed curve resolves into a trefoil knot. Using the same techniques, it is easily shown that if $A$ is a nontrivial open subarc of a shadow $S$, then $A$ resolves into a $1$-tangle whose closure is a trefoil knot. From this it follows that if $S$ is an $m$-nontrivial shadow, then $S$ resolves into a connected sum of $m$ trefoil knots.
\end{proof}

\subsection{{A kind of shadow inspired by Figure~\ref{fig:fir}}}
{To formally identify the fourth kind of shadow that plays a major role in the proof of Theorem~\ref{thm:main}, we start with an observation. We refer the reader back to Figure~\ref{fig:fir} for an illustration. Let $1$ be a crossing of a shadow $S$, and let $S_1,S_2'$ be the subshadows of $S$ based at $1$. That is, $S=S_1{\oplus_{1}} S_2'$. Suppose now that $S_2'$ (seen as a shadow on its own) has a crossing $2$, and let $S_2,S_3$ be the subshadows of $S_2'$ based at $2$, so that $S_2'=S_2{\oplus_{2}} S_3$.}

{Thus $S=S_1{\oplus_{1}} (S_2{\oplus_{2}}S_3)$. If we now go back to $S$, and consider the crossing $2$, we find that $S=S_1'{\oplus_{2}} S_3$, where $S_1'$ is precisely $S_1{\oplus_{1}}S_2$. Thus we can unambiguously write $S=S_1{\oplus_{1}}S_2{\oplus_{2}}S_3$, as $S_1{\oplus_{1}}(S_2{\oplus_{2}}S_3) = (S_1{\oplus_{1}}S_2){\oplus_{2}}S_3$.
}

{An iterative application of this observation yields that if $1,\ldots,m-1$ are crossings of a shadow $S$, then there exist shadows $S_1,\ldots,S_m$ such that $S=S_1{\oplus_{1}} \cdots {\oplus_{m-1}} S_m$. We say that $S$ {\em decomposes} into the shadows $S_1,\ldots,S_m$.
} 


\begin{definition}
{Let $m\ge 2$ be an integer. A shadow is $m$-{\em decomposable} if it decomposes into $m$ shadows, each of which resolves into a trefoil knot.}
\end{definition}

\ignore{
\begin{definition}
 For each integer $m\ge 2$, a shadow $S$ is $m$-{\em decomposable} if it has a crossing point $x$ such that one of the subshadows of $S$ based at $x$ resolves into a trefoil knot, and the other subshadow is $(m-1)$-decomposable. 
\end{definition}
}
\begin{lemma}\label{lem:second}
{For each integer $m\ge 2$}, every $m$-decomposable shadow resolves into a connected sum of $m$ trefoil knots.
\end{lemma}

As we will see below, Lemma~\ref{lem:second} follows easily from the next remark. 

\begin{observation}\label{obs:seci}
{Suppose that $S=S_1\oplus_{1} S_2$, and that $S_1$ resolves into a trefoil knot, and that $S_2$ resolves into a knot $K$.} Then $S$ resolves into a connected sum of $K$ with a trefoil knot.
\end{observation}

\begin{proof}As shown in Figure~\ref{fig:nef}(b), we obtain a resolution of $S$ by combinining the resolution of $S_1$ into a trefoil knot $T$ and the resolution of $S_2$ into a knot $K$, prescribing that each crossing between $T$ and $K$ is an overpass for the strand in $T$. 

In this resolution $K'$ of $S$, the crossing $1$ is nugatory. As illustrated in (c), twisting $T$ around $1$ shows that $K'$ is a connected sum of $K$ with a trefoil knot.
\end{proof}

\begin{figure}[ht!]
\centering
\scalebox{0.5}{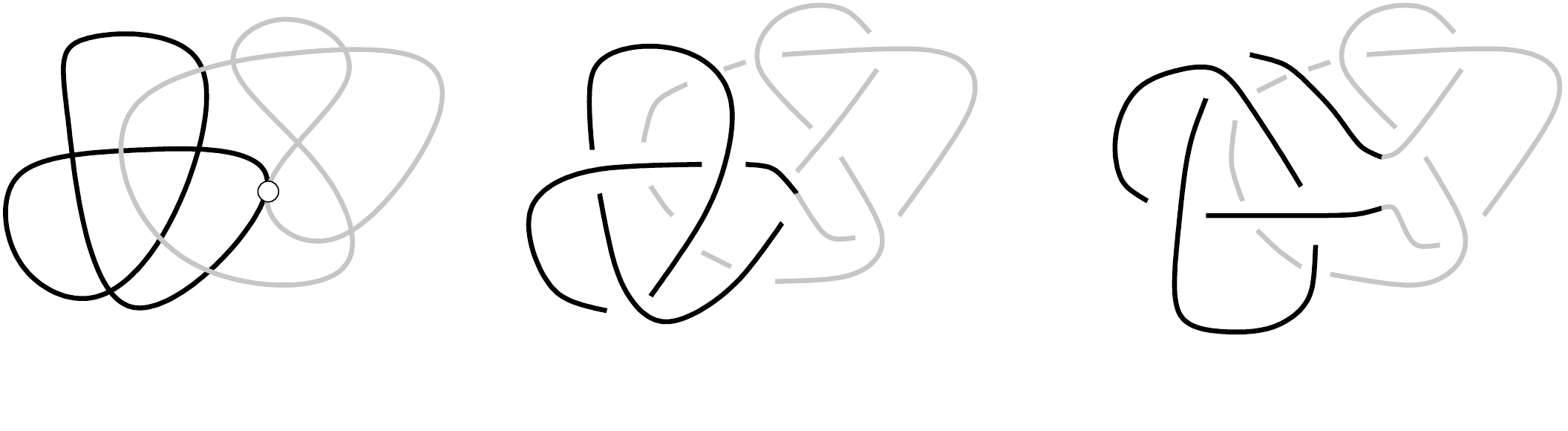}
\caption{Illustration of the proof of Observation~\ref{obs:seci}.} 
\label{fig:nef}
\end{figure}

\begin{proof}[Proof of Lemma~\ref{lem:second}]
The lemma follows immediately from the definition of an $m$-decomposable shadow and an inductive application of Observation~\ref{obs:seci}.
\end{proof}

\ignore{
\begin{figure}[ht!]
\centering
\scalebox{0.68}{\input{bt4.pdf_t}}
\caption{A shadow of a torus knot $T_{2,m+1}$ (left) and a shadow of a twist knot $T_m$ (right).}
\label{fig:overc}
\end{figure}
}

\subsection{The Structure Lemma and proof of Theorem~\ref{thm:main}}

The final ingredient for the proof of Theorem~\ref{thm:main} is the following result, whose proof is given in Section~\ref{sec:structure}.

\begin{lemma}\label{lem:structure}
For each even integer $m\ge 2$, there is an integer $n$ with the following property. Every reduced shadow with at least $n$ crossings is either $m$-nontrivial, or $m$-decomposable, or $m$-consistent, or $m$-reverse.
\end{lemma}

\begin{proof}[Proof of Theorem~\ref{thm:main}]
It follows from Lemmas~\ref{lem:third}, \ref{lem:fourth},~\ref{lem:first}, ~\ref{lem:second}, and~\ref{lem:structure}.
\end{proof}



\section{Proofs of Lemmas~\ref{lem:third} and~\ref{lem:fourth}}\label{sec:thirdfourth}

\ignore{
Before proceeding to the proofs of the lemmas, we give an informal overview of the main ideas behind the proof of Lemma~\ref{lem:third}; the proof of Lemma~\ref{lem:fourth} follows a very similar strategy. 
%
%
We refer the reader to Figure~\ref{fig:101}(a) below. If $S$ is an $m$-consistent shadow, then using standard arguments we may assume that the layout of $S$ is as illustrated in that figure. That is, (i) with the exception of a crossing $0$, all crossings of $S$ are inside a disk $\Delta$ in the $xy$-plane; and (ii) the part $U$ of $S$ inside $\Delta$ is a {\em tangle shadow} (formal definitions will be given shortly) with the property that the arcs $L$ and $R$ cross each other at points $1,\ldots,m$, and as we traverse the arc $L$ from bottom to top, and also as we traverse the arc $R$ from bottom to top, we encounter these crossings in precisely this order. 
The workhorse behind the proof of Lemma~\ref{lem:third} is the following: such a tangle shadow $U$ is a projection of a tangle isotopic to the tangle shown in Figure~\ref{fig:wh1}(a). For the final step in the proof of the lemma, we observe that the part of $S$ outside $\Delta$ is a projection of the piecewise linear strings outside the cylinder in Figure~\ref{fig:101}(b). Combining these arguments, it follows that $S$ resolves into a knot $K$ that is isotopic to the knot illustrated in Figure~\ref{fig:wh2}, which is clearly a torus knot $T_{2,m+1}$.
}


In the proofs of Lemmas~\ref{lem:third} and~\ref{lem:fourth} we make essential use of tangles. We adopt the notion that a {\em tangle} is the disjoint union of two {\em strings} (homeomorphic images of $[0,1]$) in the cylinder $\Delta\times[0,3]$, where $\Delta$ is the disk in the $xy$-plane of radius $\sqrt{2}$ centered at the origin. Admittedly, this choice of host cylinder for tangles may seem unnatural, but it will be very convenient for illustration purposes. 

\begin{figure}[H]
\centering
\scalebox{0.32}{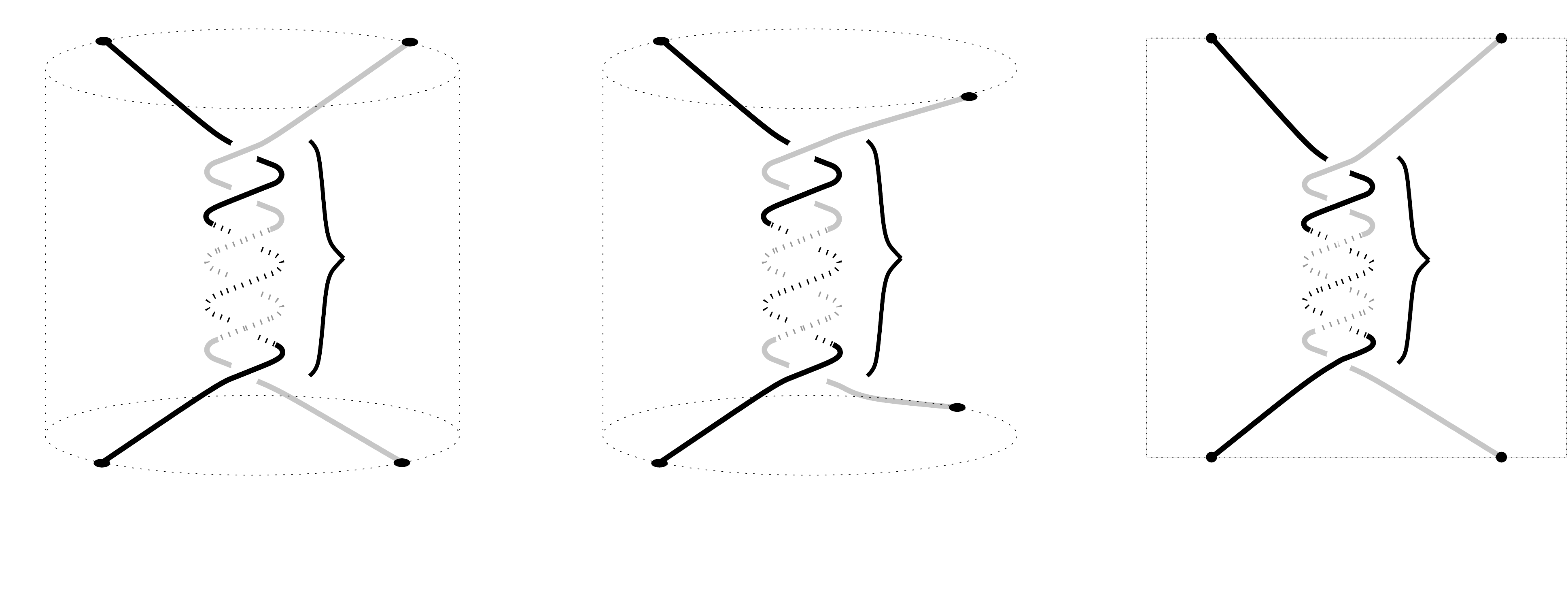}
\caption{In (a) we illustrate a tangle of Type I, and in (b) a tangle of Type II. Both tangles have the braid diagram in (c).}
\label{fig:wh1}
\end{figure}

All tangles we consider are $z$-{\em monotone}, that is, for each $c\in[0,3]$ the plane $z=c$ intersect each string in exactly one point. Moreover, we only consider two particular types of tangles, illustrated in Figure~\ref{fig:wh1}. A tangle is {\em of Type I} if it consists of a string $\lambda$ with endpoints $(-1,-1,0)$ and $(-1,1,3)$, and a string $\rho$ with endpoints $(1,-1,0)$ and $(1,1,3)$, and it is {\em of Type II} if it consists of a string $\lambda$ with endpoints $(-1,-1,0)$ and $(-1,1,3)$, and a string $\rho$ with endpoints $(1,1,0)$ and $(1,-1,3)$. 

\ignore{Two tangles $T_1,T_2$ are {\em equivalent} if there is an isotopic deformation of the cylinder, fixed on its boundary, which deforms $T_1$ into $T_2$. We note that, evidently, a tangle of Type I cannot be equivalent to a tangle of Type II.}

The {\em shadow} $U$ of a tangle $T$ is its projection onto the $xy$-plane, without over/under information at the crossings. Thus, regardless of whether $T$ is of Type I or II, $U$ consists of an arc $L$ (the projection of $\lambda$) with endpoints $(-1,-1)$ and $(-1,1)$ and an arc $R$ (the projection of $\rho$) with endpoints $(1,-1)$ and $(1,1)$. We refer the reader to Figure~\ref{fig:101}(a) (the part contained in $\Delta$) for an illustration of a tangle shadow.

{The {\em vertical diagram} of a tangle (or of a knot) is its projection onto the plane $y=2$, with over/under information at each crossing. Since every tangle $T$ we consider is $z$-monotone, the vertical diagram of $T$ is a braid diagram. This is {\em the braid diagram of $T$}. We define the {\em linking index} $\lk(T)$ of $T$ as the linking index of its braid diagram~\cite{muras2}. The tangles in Figure~\ref{fig:wh1}(a) and (b) have the same braid diagram, shown in Figure~\ref{fig:wh1}(c), and so the linking index of each of these tangles is $m/2$.}

\medskip
\noindent{\bf Remark. } {In all illustrations of vertical diagrams, the indicated coordinates of points are the $x$- and $z$-coordinates of these points, as they all have $y$-coordinate $2$.}
\vglue 0.15 cm
Our interest in tangles lies in Proposition~\ref{pro:key} below, which is the workhorse behind the proofs of Lemmas~\ref{lem:third} and~\ref{lem:fourth}. We use the following terminology, motivated by the definition of an $m$-consistent or $m$-reverse knot shadow. 

Let $m\ge 2$ be an even integer. We say that a tangle shadow $U$ has {\em rank $m$} if there exist crossings $1,\ldots,m$ between the arcs $L$ and $R$ of $U$, such that as we traverse $L$ from $(-1,-1)$ to $(-1,1)$, and also as we traverse $R$ from $(1,-1)$ to $(1,1)$, we encounter these crossings in precisely this order. In Figure~\ref{fig:101}(a) we illustrate a tangle shadow of rank $m$ (inside the disk $\Delta$). On the other hand, if as we traverse $L$ from $(-1,-1)$ to $(-1,1)$ we encounter these crossings in this order, but as we traverse $R$ from $(1,-1)$ to $(1,1)$ we encounter these crossings in the reverse order $m,\ldots,1$, then we say that $U$ has {\em rank $-m$}. In Figure~\ref{fig:c1}(a) we illustrate a tangle shadow of rank $-m$.

\begin{proposition}\label{pro:key}
Let $U$ be a tangle shadow, and let $m\ge 2$ be an even integer. If $U$ has rank $m$ (respectively, $-m$) then $U$ is the shadow of a tangle $T$ of Type I (respectively, of Type II) such that $|\lk(T)|=m/2$.
\end{proposition}

We defer the proof of this statement for the moment, and give the proofs of Lemmas~\ref{lem:third} and~\ref{lem:fourth}. We have a final observation before proceeding to the proofs. Let us say that two shadows $S,S'$ are {\em analogous} if $S$ resolves into a knot $K$ if and only if $S'$ resolves into a knot $K'$ isotopic to $K$. The observation we use is that if $x$ is a crossing in a shadow $S$, then using a standard Riemann stereographic projection argument we may turn $S$ into an analogous shadow $S'$ in which $x$ is incident with the unbounded face.

\ignore{
\begin{figure}[H]
\centering
\scalebox{0.38}{\input{jau1.pdf_t}}
\caption{These tangles play a key role in the proofs of Lemmas~\ref{lem:third} and~\ref{lem:fourth}.}
\label{fig:wh1}
\end{figure}
}

\ignore{
\begin{figure}[ht!]
\centering
\scalebox{0.38}{\input{qr1.pdf_t}}
\caption{In (a) we illustrate a tangle, and in (b) and (c) we illustrate the shadow and diagram of $T$, respectively, which result by projecting $T$ onto the $xy$-plane.}
\label{fig:att}
\end{figure}
}

\begin{proof}[Proof of Lemma~\ref{lem:third}]
Let $S$ be an $m$-consistent shadow on the $xy$-plane, for some even integer $m\ge 2$. We recall that this means that $S$ has a crossing $0$, such that the subshadows $S_1,S_2$ based at $0$ satisfy that there are crossings $1,2,\ldots,m$ between $S_1$ and $S_2$ that we encounter in this order as we traverse each of $S_1$ and $S_2$, starting at $0$. Using the observation mentioned before this proof, we may assume that $0$ is incident with the unbounded face of $S$. Performing a suitable self-homeomorphism of the plane, we may further assume that the layout of $S$ is as shown in Figure~\ref{fig:101}(a). In particular, with the exception of $0$, all crossings of $S$ are contained inside the disk $\Delta$. In this illustration, $S_1$ is the black subshadow and $S_2$ is the gray subshadow. 

\begin{figure}[ht!]
\centering
\scalebox{0.26}{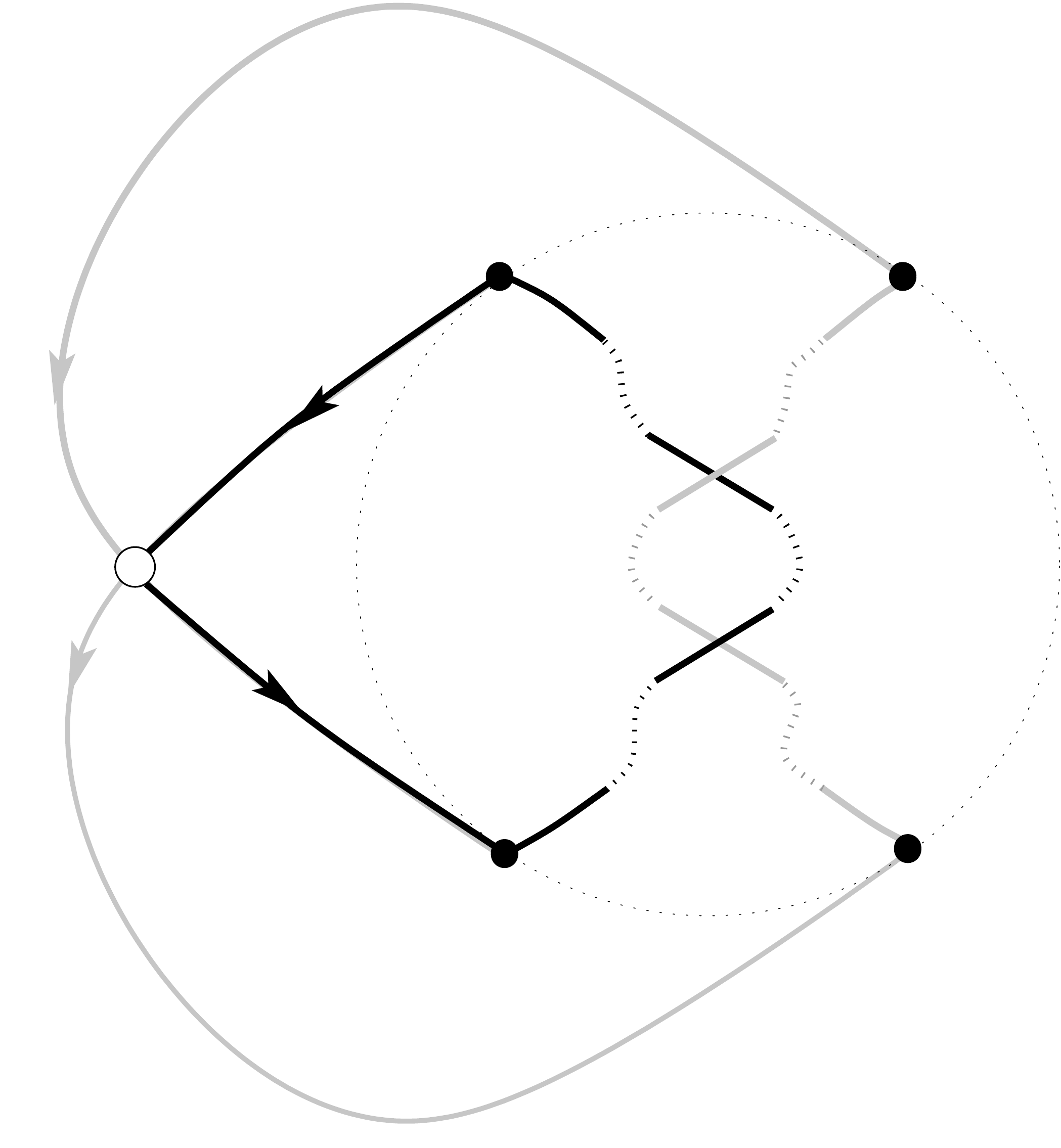}\hglue 1 cm
\scalebox{0.27}{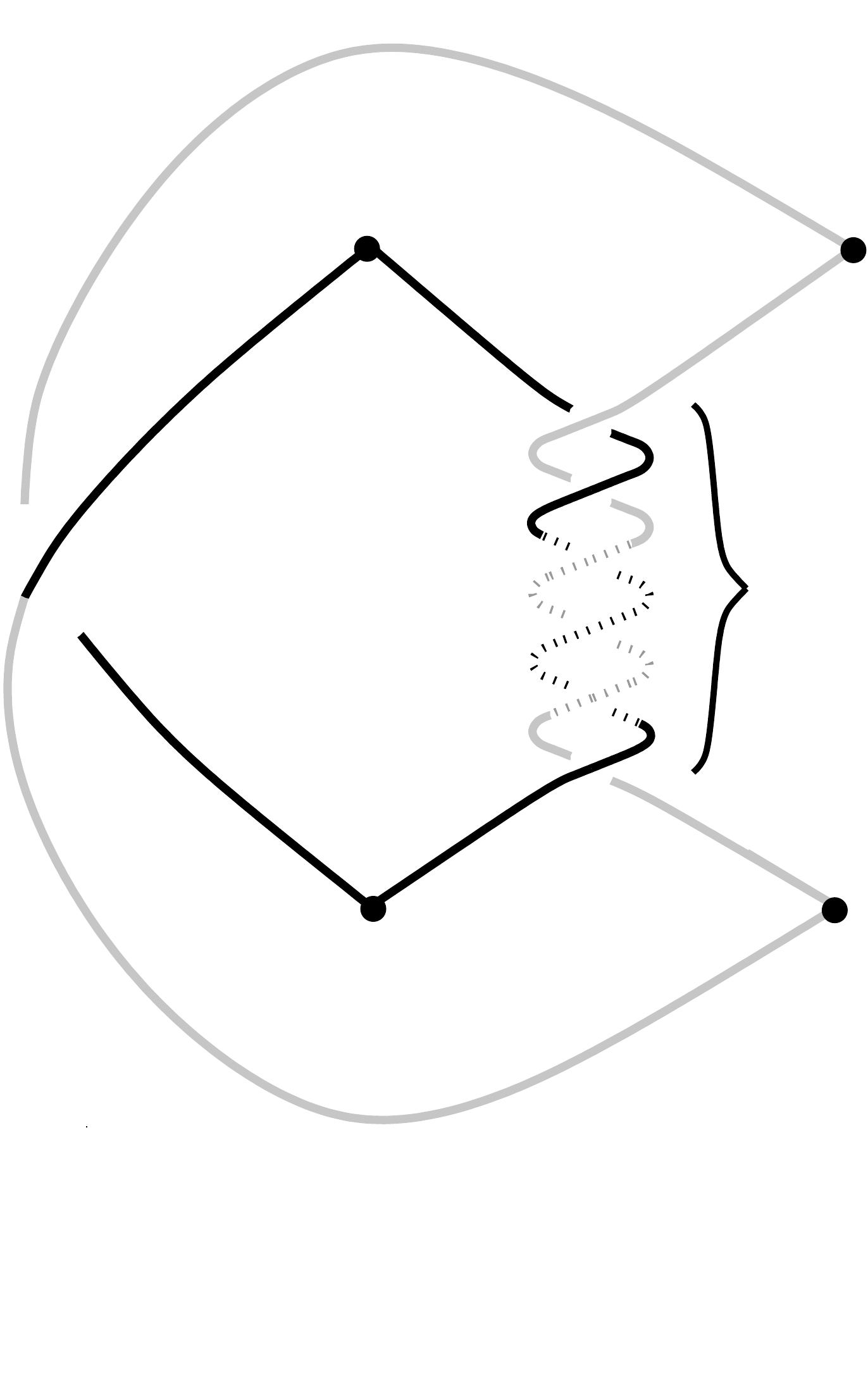}
\caption{Illustration of the proof of Lemma~\ref{lem:third}.}
\label{fig:101}
\end{figure}

The $m$-consistency of $S$ implies that the part $U$ of $S$ inside $\Delta$ is a tangle shadow of rank $m$. Thus it follows from Proposition~\ref{pro:key} that $U$ is the shadow of a tangle $T$ of Type I such that $|\lk(T)|=m/2$. We assume that $\lk(T)=m/2$, as the arguments in the case $\lk(T)=-m/2$ are totally analogous.

{It is easy to see that there exist strings $\alpha$ and $\beta$ in $3$-space, disjoint from the interior of the cylinder $\Delta\times[0,3]$, such that (i) the endpoints of $\alpha$ (respectively, $\beta$) are $(-1,-1,0)$ and $(-1,1,3)$ (respectively, $(1,-1,0)$ and $(1,1,3)$); (ii) the projection of $\alpha\cup\beta$ onto the $xy$-plane is $S\setminus U$; and (iii) the vertical projections of $\alpha$ and $\beta$ are the strands $a$ and $b$, respectively, shown in Figure~\ref{fig:101}(b).}

{Let $K$ be the knot obtained by adding $\alpha\cup\beta$ to the tangle $T$. Since $U$ is the shadow of $T$, and $S\setminus U$ is the shadow of $\alpha\cup\beta$, it follows that $S$ resolves into $K$. Consider now the vertical diagram $D$ of $K$. The part of $D$ that corresponds to $\alpha$ and $\beta$ are the strands $a$ and $b$; the rest of $D$ is the braid diagram of $T$. Since $\lk(T)=m/2$, a sequence of Reidemeister moves of Type II on this braid diagram takes this part of $D$ into the braid diagram shown in Figure~\ref{fig:101}(b). Thus $D$ is equivalent to the diagram in Figure~\ref{fig:101}(b), which is a diagram of a torus knot $T_{2,m+1}$. We conclude that $S$ resolves into a torus knot $T_{2,m+1}$.}
\end{proof}


\begin{proof}[Proof of Lemma~\ref{lem:fourth}]
Let $S$ be an $m$-reverse shadow, where $m\ge 2$ is an even integer. Similarly as in the proof of Lemma~\ref{lem:third}, we may assume that the layout of $S$ is as shown in Figure~\ref{fig:c1}(a). In this case, since $S$ is $m$-reverse it follows that the part $U$ of $S$ inside $\Delta$ is a tangle shadow of rank $-m$. Thus it follows from Proposition~\ref{pro:key} that $U$ is the shadow of a tangle $T$ of Type II such that $|\lk(T)|=m/2$. As in the proof of Lemma~\ref{lem:third} we assume that $\lk(T)=m/2$, as the arguments in the case $\lk(T)=-m/2$ are totally analogous.

\begin{figure}[ht!]
\centering
\scalebox{0.25}{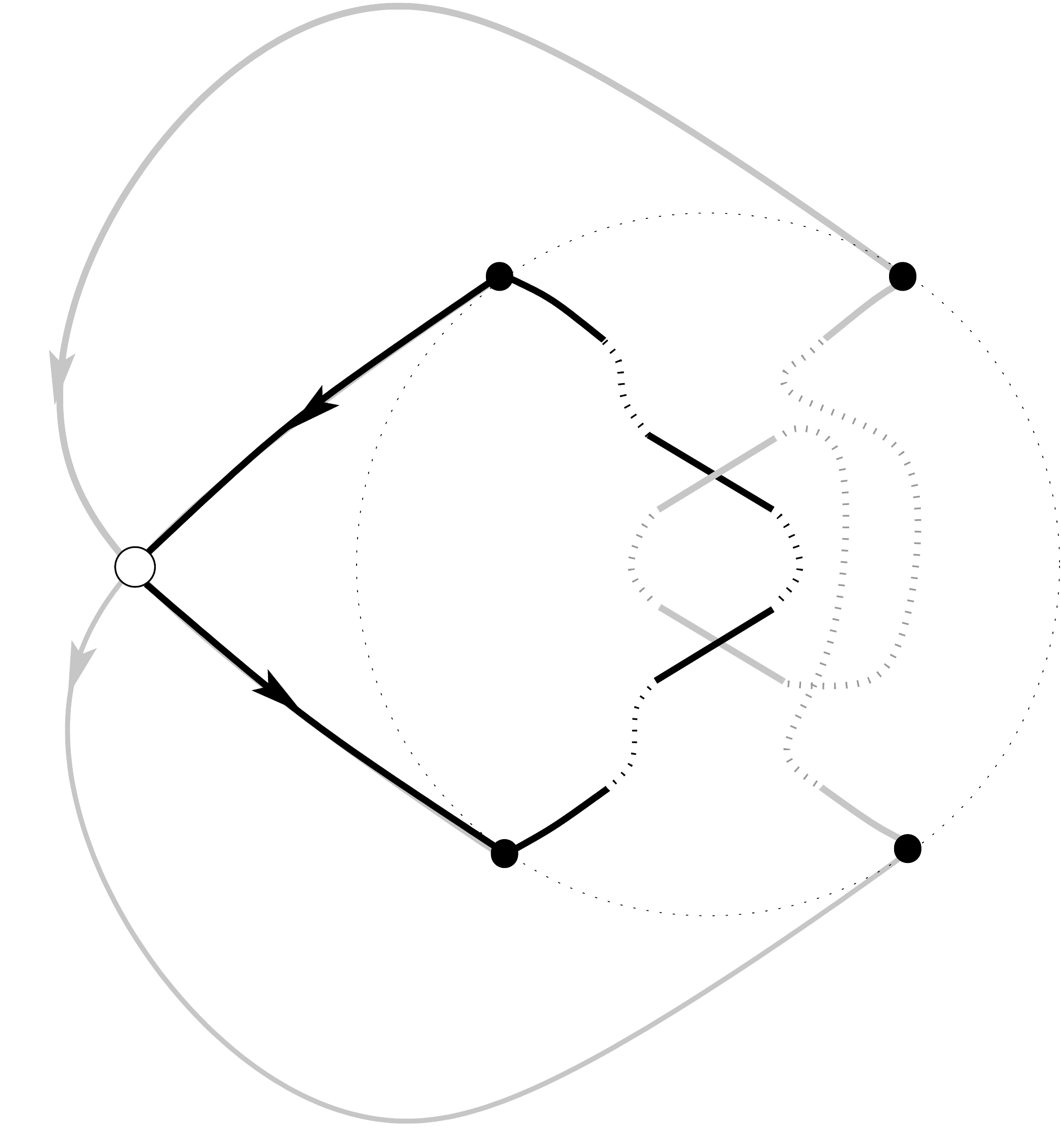}\hglue 0.8 cm
\scalebox{0.29}{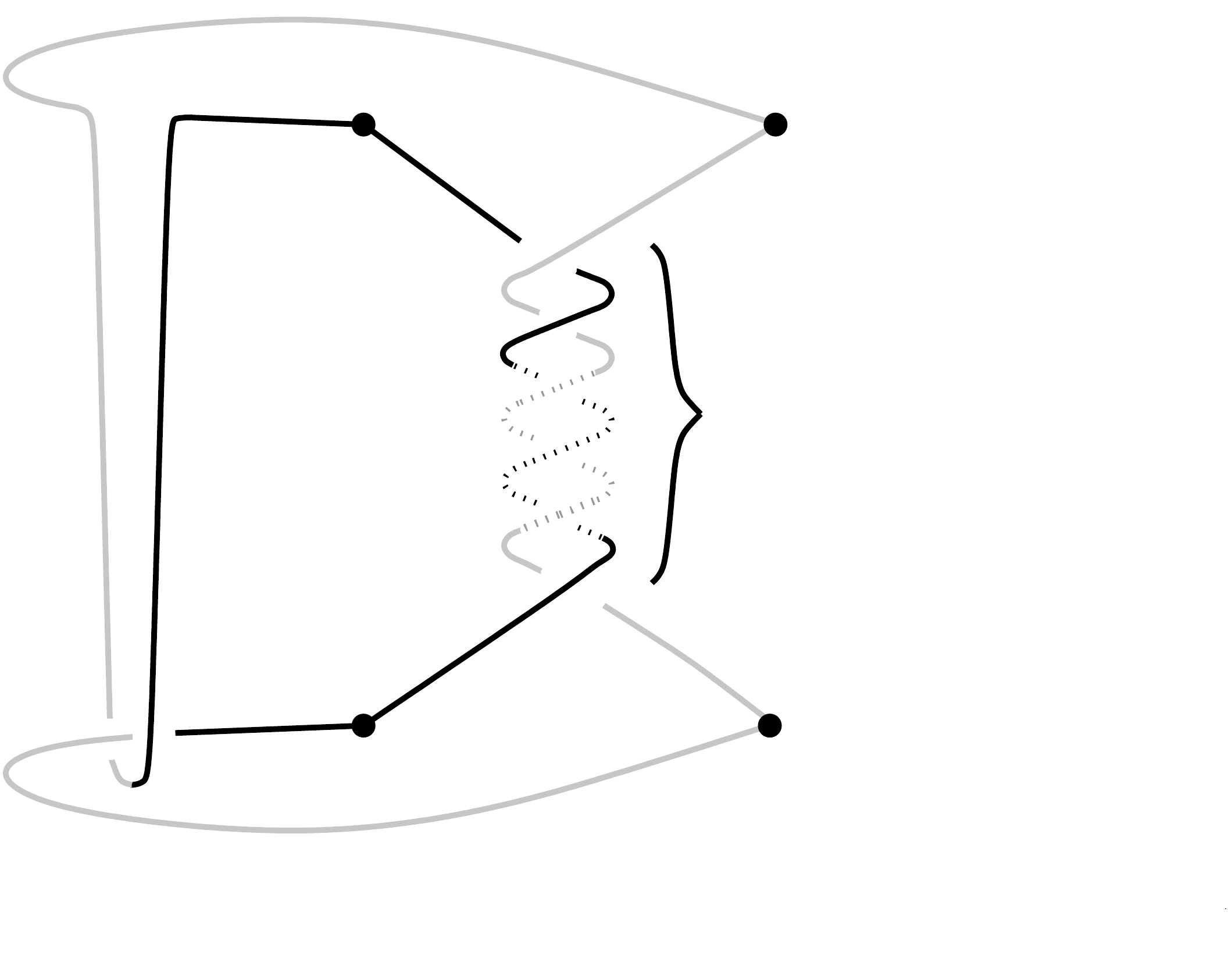}
\caption{Illustration of the proof of Lemma~\ref{lem:fourth}.}
\label{fig:c1}
\end{figure}

{It is easy to see that there exist strings $\alpha$ and $\beta$ in $3$-space, disjoint from the interior of the cylinder $\Delta\times[0,3]$, such that (i) the endpoints of $\alpha$ (respectively, $\beta$) are $(-1,-1,0)$ and $(-1,1,3)$ (respectively, $(1,1,0)$ and $(1,-1,3)$); (ii) the projection of $\alpha\cup\beta$ onto the $xy$-plane is $S\setminus U$; and (iii) the vertical projections of $\alpha$ and $\beta$ are the strands $a$ and $b$, respectively, shown in Figure~\ref{fig:c1}(b).}

{Let $K$ be the knot obtained by adding $\alpha\cup\beta$ to $T$. Using analogous arguments as in the last part of the proof of Lemma~\ref{lem:third}, it follows that $S$ resolves into a twist knot $T_m$.}
\end{proof}

\begin{proof}[Proof of Proposition~\ref{pro:key}]
We give the proof of the proposition for tangle shadows that have rank $m$, as the proof for tangle shadows with rank $-m$ is totally analogous. 

Let $m\ge 2$ be an even integer, and let $U$ be a tangle shadow of rank $m$, as illustrated in Figure~\ref{fig:to1}(a). {Let $A$ and $B$ be the arcs also illustrated in that figure. Note that $S=U\cup A\cup B$ is a shadow of a $2$-component link. }

{It is easy to see that there exist strings $\alpha$ and $\beta$ in $3$-space, disjoint from the interior of the cylinder $\Delta\times[0,3]$, such that (i) the endpoints of $\alpha$ (respectively, $\beta$) are $(-1,-1,0)$ and $(-1,1,3)$ (respectively, $(1,-1,0)$ and $(1,1,3)$); (ii) the projections of $\alpha$ and $\beta$ onto the $xy$-plane are $A$ and $B$, respectively; and (iii) the vertical projections of $\alpha$ and $\beta$ are the strands $a$ and $b$, respectively, shown in Figure~\ref{fig:to1}(b).}

\begin{figure}[ht!]
\centering
\hglue -0.7 cm \scalebox{0.27}{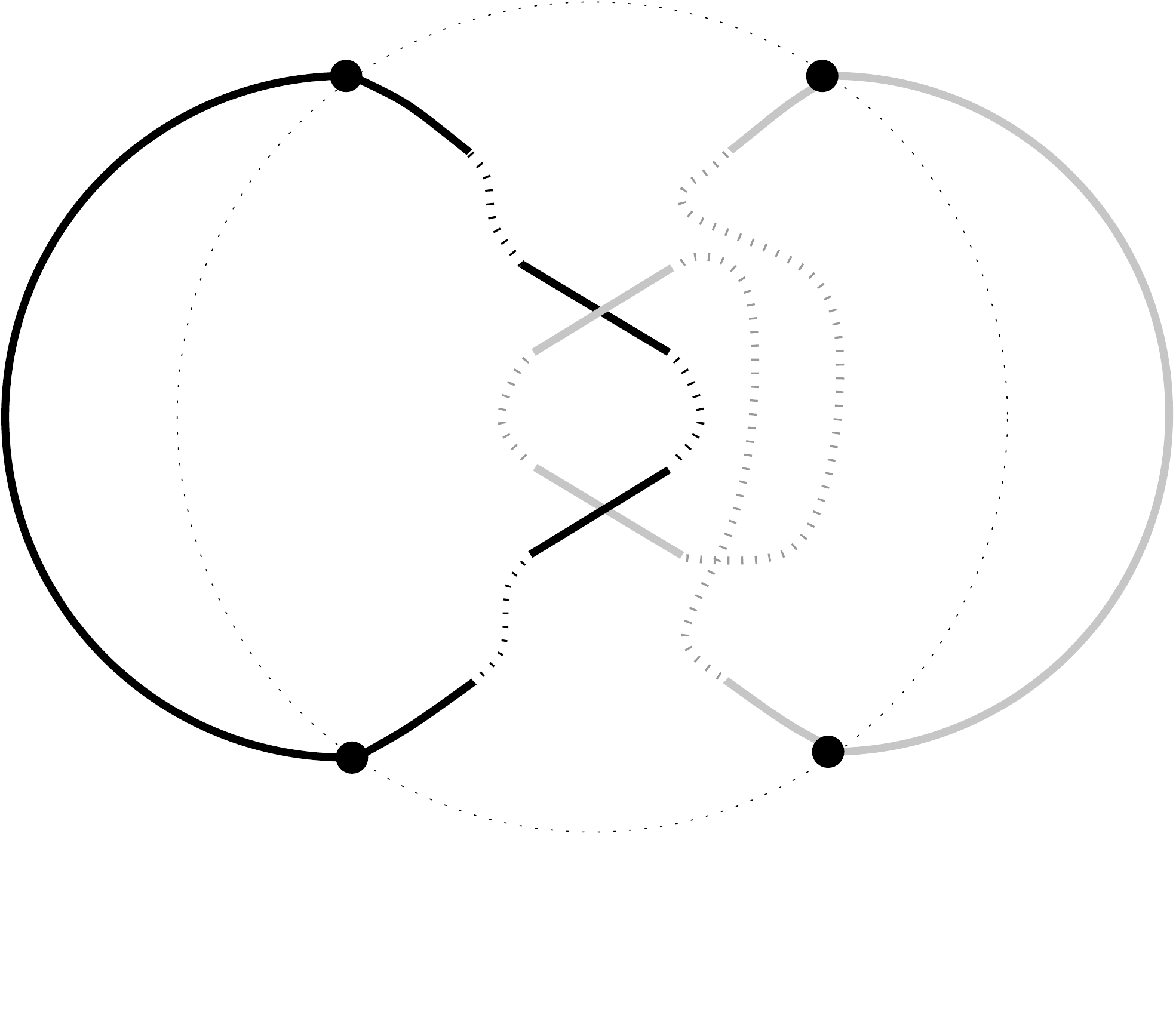}\hglue 1 cm 
\scalebox{0.24}{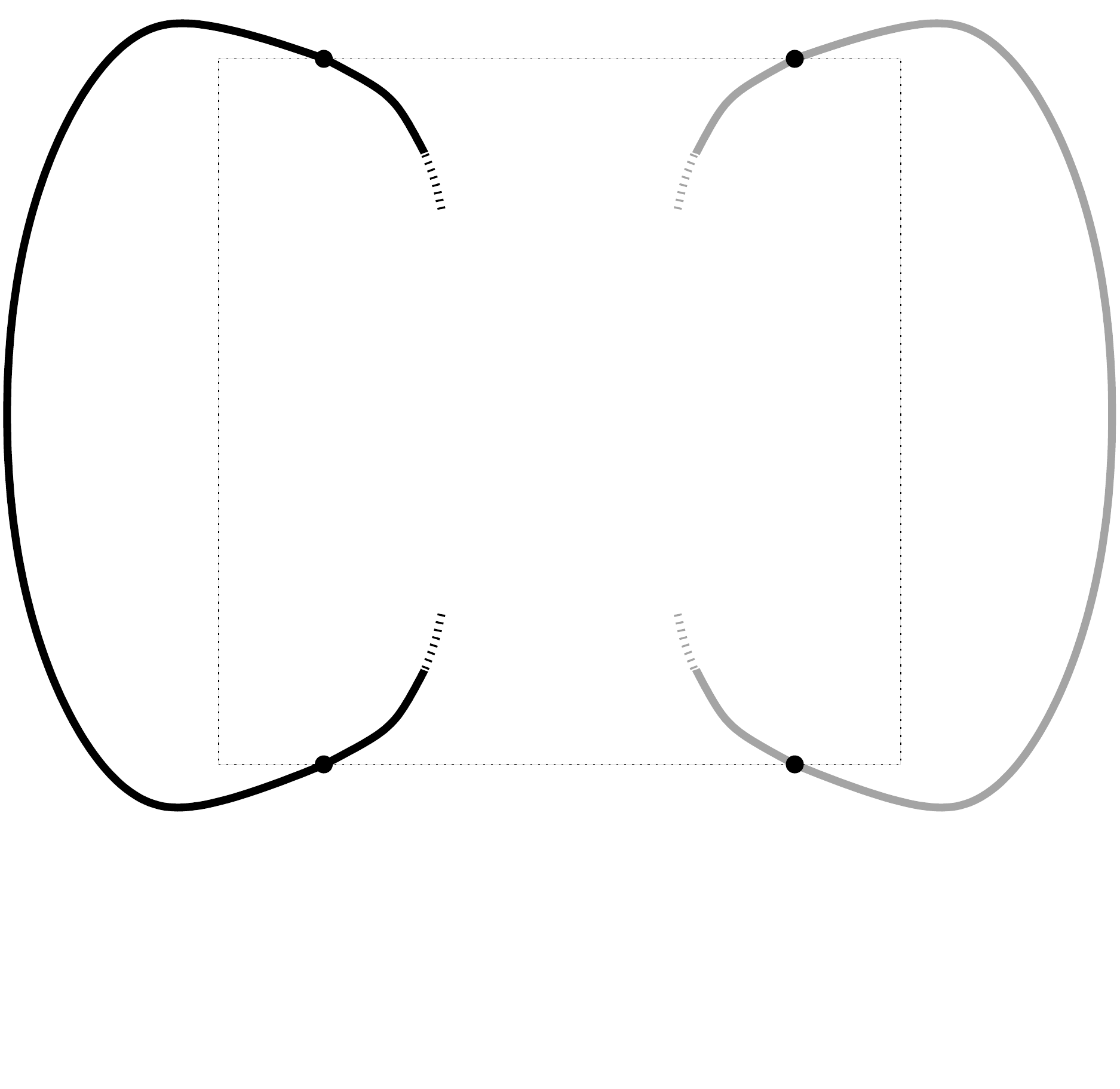}
\caption{Illustration of the proof of Proposition~\ref{pro:key}.} 
\label{fig:to1}
\end{figure}

{The strategy to prove the proposition is to show that $S$ is the shadow of a link that satisfies certain properties. We start by letting $\mm$ be the set of all links $M$ that have $S$ as their shadow, and such that: (i) the part of $M$ that projects to $U$ is a tangle $T$ of Type I; and (ii) the part of $M$ that projects to $A\cup B$ is $\alpha\cup\beta$.}

Using that $U$ has rank $m$, a straightforward adaptation of the techniques and arguments in~\cite[Algorithm 4]{taniyama2} and~\cite[Proof of Theorem 4]{taniyama2} shows that there is a link $M_0\in\mm$ whose linking number $\Lk(M_0)$ satisfies $|\Lk(M_0)|=m/2$. (We use $\Lk(M_0)$ to denote the linking number of a link $M_0$, to distinguish it from the linking index $\lk(T)$ of a tangle $T$). 

Let $T_0$ be the part of $M_0$ that is a tangle of Type I whose shadow is $U$. Let $D$ be the vertical diagram of $M_0$ (see Figure~\ref{fig:to1}(b)). The vertical projections of $\alpha$ and $\beta$ in $D$ do not intersect each other, and they do not intersect the projection of $T_0$ (which is the braid diagram of $T_0$). Thus all the crossings in $D$ are the crossings in the braid diagram of $T_0$. Therefore $\Lk(M_0)=\lk(T_0)$, and so $|\lk(T_0)|=m/2$.
\end{proof}

\section{The four relevant types of shadows in terms of Gauss codes}\label{sec:gauss}

In this section we take a first step toward the proof of Lemma~\ref{lem:structure}, finding conditions, in terms of Gauss codes, that guarantee that a shadow is $m$-nontrivial, or $m$-decomposable, or $m$-consistent, or $m$-reverse. This will turn the proof of Lemma~\ref{lem:structure} into a purely combinatorial problem.



We start with a brief review of the notion of a Gauss code of a shadow $S$. Label the crossing points of $S$, and let $p$ be an arbitrary noncrossing point of $S$. The {\em Gauss code} $\omega$ of $S$ {\em starting at $p$} is the word obtained by traversing $S$ and noting each crossing point we encounter. Thus every label occurs exactly twice in $\omega$: if the crossings of $S$ are labelled $1,\ldots,n$, then $\omega$ is a permutation of the multiset $\{1,1,\ldots,n,n\}$. 


\ignore{
\begin{figure}[ht!]
\centering
\hglue -0.8 cm \scalebox{0.5}{\input{aga2.pdf_t}}
\caption{The Gauss code of this shadow, starting at $p$, is $1\, 4\, 5\, 2\, 4\, 1\, 3\, 5\, 2\, 3$.} 
\label{fig:1f}
\end{figure}
}

We adopt the following standard terminology. A {\em substring} of a word $a_1a_2\cdots a_t$ is a word of the form $a_i a_{i+1}\cdots a_{j-1} a_{j}$, for some $1\le i\le j\le t$. A {\em subword} of $a_1a_2\cdots a_t$ is a word of the form $a_{i_1} a_{i_2} \cdots a_{i_j}$, where $1 \le i_1 < i_2 < \cdots < i_j \le t$. We adhere to the convention to use $\sigma\Ds\omega$ to denote that $\sigma$ is a subword of $\omega$.


{We start by finding a condition for a Gauss code $\omega$ that guarantees that its corresponding shadow $S$ is $m$-nontrivial.} We say that a substring $\alpha$ of $\omega$ is {\em good} if it contains distinct symbols $a_i,a_j,a_k$, such that the following hold:

\smallskip

\begin{description}
\item[(1)] no symbol of $\omega$ has both occurrences in between the two occurrences of $a_i$; and
\item[(2)] $\alpha$ contains both occurrences of each of $a_i, a_j$ and $a_k$, and each of $a_j$ and $a_k$ occurs exactly once in between the occurrences of $a_i$.
\end{description}

\smallskip

Let $A$ be a open subarc of $S$, and let $\alpha$ be the substring that is the part of $\omega$ that corresponds to the traversal of $A$. Suppose that $\alpha$ is good. Then {(1)} implies that $a_i$ is the root of a loop $L$ contained in $A$, and {(2) implies} that $L$ is crossed at least twice in $A$. That is, $A$ is a nontrivial open subarc of $S$. 

{We say that a Gauss code is $m$-{\em good} if it is the concatenation of $m$ good substrings. The observation in the previous paragraph implies the following.}

\begin{fact}\label{fac:non}
Let $S$ be a shadow, and let $\omega$ be a Gauss code of $S$. If $\omega$ {is $m$-good}, then $S$ is $m$-nontrivial.
\end{fact}




\ignore{
\begin{figure}[ht!]
\centering
\scalebox{0.75}{\input{al11.pdf_t}}
\caption{The Gauss code $\omega$ of the shadow $S$ on the left hand side, starting at $p$, is $4\,1\,3\,2\,1\,a\,5\,\,11\,\,10\,\,8\,9\,6\,\,11\,\,10\,\,7\,9\,a\,4\,2\,3\,8\,7\,6\,5$. Let $\alpha=4\,1\,3\,2\,1\,$, $\gamma=5\,\,11\,\,10\,\,8\,9\,6\,\,11\,\,10\,\,7\,$ $9$, and $\beta=4\,2\,3\,8\,7\,6\,5$. Then $\omega=\alpha\, a\, \gamma\, a\, \beta$. Let $S_1$ (center) and  $S_2$ (right) be the subshadows of $S$ based at $x$. The subword $\mea{\alpha\beta}=4\,1\,3\,2\,1\, 4\,2\,3$ is a Gauss code of $S_1$, and $\mea{\gamma}=11\,\,10\,\,9\,\,11\,\,10\,\,9$ is a Gauss code of $S_2$.}
\label{fig:al}
\end{figure}
}

{To investigate the Gauss codes of $m$-consistent, $m$-reverse, and $m$-nontrivial shadows,} we will use the following terminology. Let $S$ be a shadow, and let $\omega$ be a Gauss code of $S$. Two symbols $a_i, a_j$ of $\omega$ form an {\em alternating pair} if either $a_i a_j a_i a_j\Ds\omega$ or $a_j a_i a_j a_i\Ds\omega$. A symbol $a$ of $\omega$ is {\em nugatory} if it corresponds to a nugatory crossing of $S$. It is easy to see that $a$ is a nugatory symbol if and only if $a$ does not form part of an alternating pair. Finally, if $\sigma\Ds\omega$, then $\mea{\sigma}$ denotes the subword of $\sigma$ obtained by eliminating the symbols that appear only once in $\sigma$. 

We make essential use of the following easy remark. 

\begin{observation}\label{obs:dosmil}
Let $S$ be a shadow, let $\omega$ be a Gauss code of $S$, and let $a$ be a crossing of $S$. Write $\omega$ as a concatenation $\alpha\, a\, \gamma\, a \, \beta$. Then $\mea{\alpha\beta}$ is a Gauss code of one  subshadow of $S$ based at $a$, and $\mea{\gamma}$ is a Gauss code of the other subshadow.
\end{observation}

We now consider $m$-consistent shadows. Let $S$ be a shadow, and let $\omega$ be a Gauss code of $S$. {We say that $\omega$ is $m$-{\em increasing} if it has symbols $a_0,a_1,\ldots,a_m$ such that $a_0 \, a_1\, \cdots\, a_m\, a_0 \, a_1\, \cdots a_m\Ds\omega$.} If $\omega$ has such a subword, then it follows from Observation~\ref{obs:dosmil} that the Gauss codes of both subshadows of $S$ based at $a_0$ have $a_1a_2\ldots a_m$ as a subword. In view of the definition of an $m$-consistent shadow, this immediately implies the following.

\begin{fact}\label{fac:con}
Let $S$ be a shadow, and let $\omega$ be a Gauss code of $S$. If $\omega$ {is $m$-increasing}, then $S$ is $m$-consistent.
\end{fact}

{We say that $\omega$ is $m$-{\em decreasing} if it has symbols $a_0,a_1,\ldots,a_m$ such that $a_0 \, a_1\, \cdots$ $ a_m a_0 \, a_m\,\cdots a_1\Ds\omega$.} If $\omega$ has such a subword, then by Observation~\ref{obs:dosmil} the Gauss code of one of the subshadows of $S$ based at $a_0$ has $a_1a_2\ldots a_m$ as a subword, and the Gauss code of the other subshadow has $a_m \ldots a_2a_1$ as a subword. The following is then an immediate consequence of the definition of an $m$-reverse shadow.

\begin{fact}\label{fac:rev}
Let $S$ be a shadow, and let $\omega$ be a Gauss code of $S$. If $\omega$ {is $m$-decreasing}, then $S$ is $m$-reverse.
\end{fact}



{We finally find a property of a Gauss code corresponding to $m$-decomposability.} We make use of the following remark. By~\cite[Theorem 3]{ams}, every shadow in which not every crossing is nugatory resolves into a trefoil knot. As we discussed above, $a$ is a nugatory crossing of a shadow $S$ if and only if $a$ does not form part of any alternating pair of symbols in the Gauss code of $S$. Thus we have the following.

\begin{observation}\label{obs:mil}
If a Gauss code $\omega$ of a shadow $S$ has an alternating pair of symbols, then $S$ resolves into a trefoil knot.
\end{observation}

{Let $m\ge 2$ be an integer.} A Gauss code $\omega$ is $m$-{\em nice} if it can be written as
\[
\omega=\alpha_1\, a_1 \, \alpha_2\, a_2 \, \cdots \, \alpha_{m-1} \, a_{m-1} \,\, \alpha_{m}\,\beta_{m}\, \, a_{m-1}\, \beta_{m-1}\, \cdots \, a_2\, \beta_2 \, a_1\, \beta_1,
\]
{where $a_1,\ldots,a_{m-1}$ are symbols, and for $i=1,\ldots,m$, the concatenation $\alpha_i\beta_i$ has an alternating pair of symbols.}

\begin{fact}\label{fac:dec}
{Let $S$ be a shadow, let $\omega$ be a Gauss code of $S$, {and let $m\ge 2$ be an integer. If $\omega$ is $m$-nice,} then $S$ is $m$-decomposable.}
\end{fact}

\begin{proof}
{The crossings $a_1,\ldots,a_{m-1}$ induce a decomposition $S_1{\oplus_{a_1}}S_2{\oplus_{a_2}}\cdots S_{m-1}$ ${\oplus_{a_{m-1}}}S_m$ of $S$. An iterative application of Observation~\ref{obs:dosmil} yields that $\mea{\alpha_i\beta_i}$ is a Gauss code of $S_i$, for $i=1,\ldots,m$. Since $\alpha_i\beta_i$ has an alternating pair for each $i=1,\ldots,m$, then $\mea{\alpha_i\beta_i}$ also has an alternating pair for each $i=1,\ldots,m$. Therefore, by Observation~\ref{obs:mil}, $S_i$ resolves into a trefoil knot for each $i=1,\ldots,m$.}
\end{proof}

\section{Proof of Lemma~\ref{lem:structure}}\label{sec:structure}

The following propositions are the workhorses behind the proof of Lemma~\ref{lem:structure}.

\begin{proposition}\label{pro:work4}
Let $\omega$ be a Gauss code of a reduced shadow $S$, and let $m\ge 2$ be an integer. Suppose that $1\,1\,2\,2\,\cdots {9m^5}\, {9m^5} \Ds \omega$. Then {$\omega$ either is $m$-increasing, or it is $m$-good, or it has symbols $a_1,\ldots,a_{3m^2}$ such that $a_1 \, \cdots\, a_{3m^2}\, a_{3m^2} \, \cdots\, a_1\Ds\omega$.}
\end{proposition}

\begin{proposition}\label{pro:work3}
Let $\omega$ be a Gauss code of a reduced shadow $S$, and let $m\ge 2$ be an integer. Suppose that $1\,2\,\cdots\, {3m^2}\, {3m^2}\, \cdots\, 2\,1\Ds\omega$. {Then $\omega$ is either $m$-decreasing or $m$-nice.}
\end{proposition}

\begin{proof}[Proof of Lemma~\ref{lem:structure}, assuming Propositions~\ref{pro:work4} and~\ref{pro:work3}]
{Let $S$ be a reduced shadow with $n$ crossings, and let $\omega$ be a Gauss code of $S$. We show that if $n$ is at least the $3$-colour Ramsey number $R(m+1,3m^2,9m^5)$, then $\omega$ is either $m$-increasing, or $m$-decreasing, or $m$-good, or $m$-nice. This implies the lemma, using Facts~\ref{fac:non},~\ref{fac:con},~\ref{fac:rev}, and~\ref{fac:dec}.}

{We first note that we may label the symbols of $\omega$ with $1,2,\ldots,n$ so that ($*$) for each $1\le i<j\le n$, the first occurrence of $i$ is before the first occurrence of $j$.}

 Let $G$ be the complete graph whose vertices are the symbols $1,2,\ldots,n$. Let $i,j\in\{1,2,\ldots,n\}$, where $i<j$. We assign colour $1$ to the edge $ij$ if the subword of $\omega$ induced by $i$ and $j$ is $ijij$, colour $2$ if this subword is $ijji$, and colour $3$ if this subword if $iijj$. By ($*$), every edge is of one of these three colours.

By Ramsey's theorem, $G$ has either (i) a complete subgraph with vertices $a_0 {<} a_1 {<}$ $ \cdots < {a_{m}}$, all of whose edges are of colour $1$; or (ii) a complete subgraph  with vertices $a_1{<}a_2{<}\cdots{<}a_{3m^2}$, all of whose edges are of colour $2$; or (iii) a complete subgraph with vertices $a_1{<}a_2{<}\cdots{<}a_{9m^5}$, all of whose edges are of colour $3$.

{If (i) holds, then $a_0 a_1 \cdots a_m a_0 a_1 \cdots a_m\Ds\omega$, and so $\omega$ is $m$-increasing. If (ii) holds, then $a_1 a_2\cdots a_{3m^2} a_{3m^2} \cdots a_2 a_1\Ds\omega$, and so by Proposition~\ref{pro:work3} then $\omega$ is either $m$-decreasing or $m$-nice. Finally, if (iii) holds then $a_1 a_1 a_2 a_2 \cdots a_{9m^5}a_{9m^5}\Ds\omega$, and so by Proposition~\ref{pro:work4} either $\omega$ is $m$-increasing or $m$-good, or it has symbols $b_1,\ldots,b_{3m^2}$ such that $b_1 \cdots b_{3m^2} b_{3m^2}\cdots b_1\Ds \omega$. In this latter case, by Proposition~\ref{pro:work3} it follows that $\omega$ is either $m$-decreasing or $m$-nice.}
\end{proof}

\begin{proof}[Proof of Proposition~\ref{pro:work4}]
The hypothesis is that $1122\cdots 9m^5 9m^5 \Ds \omega$. We start by noting that we may assume that ($*$) for each $i=1,\cdots,9m^5$, there is no symbol $b$ in $\omega$ such that $i\, b\, b\, i$ is a subword of $\omega$.

Write $\omega$ as a concatenation $\alpha_1 \alpha_2 \cdots \alpha_{m}$, where for each $i=1,\ldots,m$, $\alpha_i$ has ${(i-1)9m^4+1}$ ${(i-1)9m^4+1}\,\,\, \cdots \,\,\, i(9m^4)\,\,\,i(9m^4)$ as a subword. If $\alpha_i$ is a good substring for every $i=1,\ldots,m$, then $\omega$ is $m$-good, and so we are done.

Thus we may assume that there is an $i\in\{1,2,\ldots,m\}$ such that $\alpha_i$ is not a good substring. To simplify the discussion, we note that $\alpha_i\,\alpha_{i+1}\cdots\,\alpha_m\, \alpha_1\,\cdots\,\alpha_{i-1}$ is also a Gauss code of $S$, and so by relabelling, if necessary, we may assume that $\alpha_1$ is not a good substring.

Recall that $\alpha_1$ contains $1\, 1\, \cdots 9m^4\, 9m^4$ as a subword. We invoke the easy fact that for every symbol $a$ in a Gauss code of a reduced shadow, there must exist two distinct symbols that occur in between the two occurrences of $a$. Thus for each $i=1,\ldots,9m^4$, there exist symbols $b_i,c_i$ such that $ib_ic_ii\Ds\omega$. Note that ($*$) implies that each of $b_i$ and $c_i$ occurs exactly once in between the two occurrences of $i$.

The hypothesis that $\alpha_1$ is not good implies that, for each $i=1,\ldots,9m^4$, there is a (at least one) $d_i\in\{b_i,c_i\}$ that only occurs once in $\alpha_1$ (namely, in between the two occurrences of $i$). Therefore there are symbols $d_1,d_2,\ldots,d_{9m^4}$ that appear exactly once in $\alpha_1$, and so each of these symbols also appears exactly once in $\alpha_2\cdots \alpha_m$. 

The Erd\H{o}s-Szekeres theorem on increasing/decreasing subsequences then implies that there are $a_1,\ldots,a_{3m^2}$ in $\{d_1,\ldots,d_{9m^4}\}$ such that either (i) $a_1a_2 \cdots $ $a_{3m^2}a_1a_2\cdots a_{3m^2}\Ds\omega$  or (ii) $a_1a_2\cdots a_{3m^2} a_{3m^2} \cdots a_2 a_1\Ds\omega$. {If (ii) holds then we are done, and if (i) holds then $\omega$ is $(3m^2-1)$-increasing, and so (since $3m^2-1 > m$) it is $m$-increasing.}
\end{proof}

\begin{proof}[Proof of Proposition~\ref{pro:work3}]
{Since $12\cdots 3m^2 3m^2\cdots 21\Ds\omega$, then we can write $\omega$ as}

\medskip

\hglue -0.45cm \scalebox{0.87}{
${\alpha_{2m} (2m) \alpha_{4m} (4m) \cdots \alpha_{2m^2-2m} (2m^2-2m)\alpha_{2m^2}\beta_{2m^2} (2m^2-2m)\beta_{2m^2-2m} \cdots (4m)\beta_{4m} (2m) \beta_{2m},}$
}

\medskip

\noindent {where the substrings $\alpha_i,\beta_i$ are uniquely determined for $i=2m,4m,\ldots,2m^2-2m$, and we set $\alpha_{2m^2}$ (respectively, $\beta_{2m^2}$) so that $(2m^2-2m+1)(2m^2-2m+2)\cdots 3m^2\Ds \alpha_{2m^2}$ (respectively, $3m^2\cdots (2m^2-2m+2)(2m^2-2m+1)\Ds\beta_{2m^2}$.}

{If $\alpha_i\beta_i$ has an alternating pair for each $i=2m,4m,\ldots,2m^2$, then this expression of $\omega$ witnesses that $\omega$ is $m$-nice, and so we are done. Thus we may assume that there is an $i\in \{2m,4m,\ldots,2m^2\}$ such that $\alpha_i\beta_i$ has no alternating pair.}

{Let $j:=i-m$, $\sigma:=(j-m+1)(j-m+2)\cdots (j-1)$, and $\tau:=(j+1)(j+2)\cdots (j+m-1)$. Note that $\sigma\,j\,\tau\Ds\alpha_i$ and $\tau^{-1}\,j\,\sigma^{-1}\Ds\beta_i$. We show that there is a symbol $b$ such that either (I) $b\sigma j b j\sigma^{-1}\Ds\omega$; or (II) $bj\,\tau\,b \tau^{-1}\,j\Ds\omega$. This will complete the proof, as each of (I) and (II) implies that $\omega$ is $m$-decreasing.}

{Since $S$ is reduced, then every symbol of $\omega$, and in particular $j$, forms part of an alternating pair. Thus there is a $b$ such that either $bjbj\Ds\omega$ or $jbjb\Ds\omega$. We may assume that the former possibility holds, as in the alternative we may work with $\omega^{-1}$, which is also a Gauss code of $S$.}

{Thus $bjbj\Ds\omega$. If $\alpha_i$ contains both occurrences of $b$, then $bjb\Ds \alpha_i$, and so (since $j$ is also in $\beta_i$) $bjbj\Ds\alpha_i\beta_i$, contradicting that $\alpha_i\beta_i$ has no alternating pair. Thus $\alpha_i$ contains at most one occurrence of $b$. Therefore either (i) the first occurrence of $b$ is to the left of $\sigma$; or (ii) the second occurrence of $b$ is to the right of $\tau$. If (i) holds then we are done, since then it follows that (I) holds. Suppose finally that (ii) holds, and that (i) does not hold (this last assumption implies that $b$ occurs in $\sigma$). The second occurrence of $b$ must then be to the left of $\tau^{-1}$, as otherwise if would necessarily be in $\tau^{-1}$, implying that $bjbj\Ds \alpha_i\beta_i$, again contradicting that $\alpha_i\beta_i$ has no alternating pair. Thus the second occurrence of $b$ is in between $\tau$ and $\tau^{-1}$, and so (II) holds. 
}
\end{proof}

\section{Open questions}

{For each reduced shadow $S$, let $f(S)$ be the number of non-isotopic knots into which $S$ resolves. The shadows $S_m$ in Figure~\ref{fig:one}(a) have $m+1$ crossings, and it is proved in~\cite{fertility} that $S_m$ resolves into a knot $K$ if and only if $K$ is a torus knot $T_{2,n}$ with crossing number at most $m+1$. Taking into account that $T_{2,n}$ is not isotopic to its mirror image $T_{2,-n}$ if $|n|>1$, it follows that $f(S_m)$ is precisely the number of crossings in $S_m$, namely $m+1$.}

{Thus for each odd integer $n\ge 3$, there is a reduced shadow $S$ with $n$ crossings such that $f(S)=n$. Is it true that for each $n\ge 3$, every reduced shadow $S$ with $n$ crossings satisfies that $f(S)\ge n$? Here is an even easier question: is there a universal constant $c>0$ such that every reduced shadow $S$ with $n$ crossings satisfies that $f(S) > c\cdot n$?} 

{What about a ``typical'' reduced shadow? Pick a shadow $S$ randomly among all reduced shadows with $n$ crossings. What is the expected value of $f(S)$? Is this number exponential, or at least superpolynomial, in $n$? The strong techniques recently developed by Chapman in~\cite{chapman} may shed light on this question.}

\ignore{
\section{Concluding remarks and open questions}\label{sec:concluding}
In~\cite{fertility}, Cantarella, Henrich, Magness, O'Keefe, Perez, Rawdon, and Zimmer proved that torus knots $T_{2,p}$ and twist knots have few {\em descendants}. In our context, they show that a shadow of a minimal crossing diagram of a torus knot $T_{2,p}$ only resolves into $(p+1)/2$ knots, and a shadow of a minimal crossing diagram of a twist knot $T_p$ (say with $p$ even) only resolves into $(p+2)/2+1$ knots. Thus there exist reduced shadows with $p$ crossings that only resolve into (roughly) $p/2$ distinct knots. Is it true that {\em every} reduced shadow with $p$ crossings resolves into at least $p/2$ distinct knots? 
%
\ignore{
\begin{conjecture}\label{conj:conj2}
There are positive constants $c,\alpha$ such that every reduced shadow with $n$ crossings resolves into at least $cn^\alpha$  distinct knot types.
\end{conjecture}
This last conjecture is not just an effort to pose a weaker, perhaps more approachable conjecture, but instead it is motivated by the following discussion. In our proof we use Ramsey's theory, and so an explicit upper bound for $n$ would be at least exponential in $m$. To keep the discussion simple, in a best case scenario our techniques would imply that every reduced shadow with at least $n$ crossings resolves into at least $\log{n}$ distinct knots. This is far away from what we propose in the previous conjecture.
We note that a polynomial upper bound for $n$ in terms of $m$ in Theorem~\ref{thm:main} would not only be in our opinion very interesting by itself, but it would also imply Conjecture~\ref{conj:conj2}. 
}
Can anything interesting be said about the fertility of a typical (random) shadow? The strong results and techniques recently developed by Chapman~\cite{chapman} may shed light on, our maybe even outright solve, the following.
\begin{question}\label{que:que3}
Pick a shadow $S$ at random among all reduced shadows with $n$ crossings. What is the expected number of distinct knots into which $S$ resolves? Is it easy to show that this expected number is at least a linear function on $n$? Is it true that this number is exponential, or at least superpolynomial, in $n$? 
\end{question}
}

\ignore{
Finally, diving further into this direction, we bring up the recent paper~\cite{uni}. In this work, Even-Zohar, Hass, Linial, and Nowik investigate infinite families of shadows that are {\em universal}, in the sense that they yield diagrams for all knots. Restricting our attention to single shadows, rather than to infinite families of shadows, let us say that a shadow $S$ is {\em $m$-universal}, for some positive integer $m$, if it resolves into all knots with crossing number at most $m$. 
We note the connection of this notion with the previously cited work by Cantarella et al.~\cite{fertility}. In that paper, it is reported that there is a shadow with $7$ crossings that resolves into {\em all} knots with crossing number $6$ or less. It is also mentioned that no shadow with $n=8,9$, or $10$ crossings resolves into all knots with crossing number $n-1$ or less, and they ask if the same is true for every $n\ge 8$. 
Most likely the answer to this last question is no, but it does not seem easy (to say the least) to show this. Going back to the notion of an $m$-universal shadow, we have a final question.
Our guess is that it is reasonably easy to show, using Chapman's techniques, that this expected number is at least polynomial in $n$. 
}

\ignore{
We can come up with many open questions around the main theme of this paper. We find the following basic question particularly interesting. Again, we feel that there should be a reasonably easy way to settle this in the affirmative, but our efforts so far have failed.
\begin{question}
Is it true that if a shadow $S$ resolves into a torus knot $T_{2,n}$ then it also resolves into a torus knot $T_{2,n,-2}$?
\end{question} 
\ignore{We finally discuss a possible extension of Theorem~\ref{thm:main} for more restricted kinds of shadows. As we noted in its statement, Theorem~\ref{thm:main} is best possible, in the sense that torus knots $T_{2,n}$, twist knots, and connected sums of trefoil knots are the only knots that are guaranteed to lie above every reduced shadow. On the other hand, the shadows that show this, namely the ones in Figure~\ref{fig:tangles2}, are ``long and thin".  
We purposefully leave the glaring feature ``long and thin" in this informal fashion, as one can think of more than one parameter that formally captures this kind of structure. One could for instance say that these shadows have small path-width~\cite{rspath},  or that they have small width in the sense of~\cite{uni}.
\begin{question}
Let $k$ be a positive (large enough, if needed) integer. Let $S$ be a shadow with path-width/width at least $k$. Can we characterize into which knots $S$ necessarily resolves?
\end{question}
}
}


\end{document}